\documentclass[a4paper, 11pt, reqno]{amsart}

\usepackage{amssymb, amsmath, latexsym, pstricks}

\newcommand{\ifpics}[1]{#1}

\newcommand\GreenL{\mathcal{L}}
\newcommand\GreenR{\mathcal{R}}
\newcommand\GreenH{\mathcal{H}}
\newcommand\GreenD{\mathcal{D}}
\newcommand\GreenJ{\mathcal{J}}
\newcommand\RelL{\mathrel{\mathcal{L}}}
\newcommand\RelR{\mathrel{\mathcal{R}}}
\newcommand\RelH{\mathrel{\mathcal{H}}}
\newcommand\RelD{\mathrel{\mathcal{D}}}
\newcommand\RelJ{\mathrel{\mathcal{J}}}
\newcommand\OrdL{\mathrel{\leq_\mathcal{L}}}
\newcommand\OrdR{\mathrel{\leq_\mathcal{R}}}
\newcommand\OrdJ{\mathrel{\leq_\mathcal{J}}}

\newcommand\trop{\mathbb{T}}
\newcommand\olt{\overline{\mathbb{T}}}
\newcommand\ft{\mathbb{FT}}

\newcommand\pft{\mathbb{PFT}}

\newtheorem{theorem}{Theorem}[section]

\newtheorem{lemma}[theorem]{Lemma}
\newtheorem{proposition}[theorem]{Proposition}

\newtheorem{corollary}[theorem]{Corollary}
\theoremstyle{definition}\newtheorem{example}[theorem]{Example}

\begin{document}

\title[Green's $\mathcal{J}$-order and the rank of tropical matrices]{Green's $\mathcal{J}$-order and \\ the rank of tropical matrices}


\maketitle

\begin{center}

    MARIANNE JOHNSON\footnote{Email \texttt{Marianne.Johnson@maths.manchester.ac.uk}.}
and
 MARK KAMBITES\footnote{Email \texttt{Mark.Kambites@manchester.ac.uk}.}

    \medskip

    School of Mathematics, \ University of Manchester, \\
    Manchester M13 9PL, \ England.

\date{\today}
\thanks{}

\end{center}

\date{\today}
\numberwithin{equation}{section}

\begin{abstract}
We study Green's $\GreenJ$-order and $\GreenJ$-equivalence for the semigroup of
all $n \times n$ matrices over the tropical semiring. We give an exact
characterisation of the $\GreenJ$-order, in terms of morphisms between
certain tropical
convex sets. We establish connections between the $\GreenJ$-order,
isometries of tropical convex sets, and various notions of rank for tropical
matrices. We also study the relationship between the relations $\GreenJ$
and $\GreenD$; Izhakian and Margolis have observed that $\GreenD \neq \GreenJ$
for the semigroup of all $3 \times 3$ matrices over the tropical semiring with
$-\infty$, but in contrast, we show that, $\GreenD = \GreenJ$ for all full matrix
semigroups over the finitary tropical semiring.
\end{abstract}

\section{Introduction}

Tropical algebra (also known as max-plus algebra or max algebra) is the
algebra of the real numbers (typically augmented with $-\infty$, and
sometimes also with $+\infty$) under the operations of addition and maximum.
It has been an
active area of study in its own right since the 1970's
\cite{CuninghameGreen79} and also has applications in diverse
areas such as analysis of discrete event
systems \cite{MaxPlus95}, combinatorial optimisation and
scheduling problems \cite{Butkovic03}, formal languages
and automata \cite{Pin98, Simon94},
 control theory \cite{Cohen99}, phylogenetics \cite{Eriksson05}, statistical
inference \cite{Pachter04}, biology \cite{Brackley11}, algebraic geometry \cite{Bergman71, Mikhalkin05, RichterGebert05} and
combinatorial/geometric group theory \cite{Bieri84}. Tropical algebra
and many of its basic properties have been independently
rediscovered many times by researchers in these fields.

Many problems arising from these application areas are naturally expressed
using (max-plus) linear equations, so much of tropical algebra concerns
matrices. From an algebraic perspective, a key object is the semigroup
of all square matrices of a given size over the tropical semiring. This
semigroup obviously plays a role analogous to that of the full matrix
semigroup over a field. Perhaps less obviously, the relative scarcity of
invertible matrices over an idempotent semifield means that much less is to be learnt
here by studying only invertible matrices. Many classical matrix problems
can be reduced to questions about invertible matrices, and hence about the
general linear group. In tropical algebra, however, there is typically no
such reduction, so problems often entail a detailed analysis of non-invertible matrices. In
this respect, the full matrix semigroup takes on the mantle of the
general linear group. Its algebraic structure has thus
been the subject of considerable study, at first on an \textit{ad hoc}
basis, but more recently moving towards a systematic understanding using
the tools of semigroup theory (see for example
\cite{dAlessandro04,K_tropd,K_puredim,Izhakian10,K_tropicalgreen,Simon94}).

Green's relations \cite{Green51,Clifford61} are five equivalence relations
($\GreenL$, $\GreenR$, $\GreenH$, $\GreenD$ and $\GreenJ$)
and three pre-orders ($\leq_\GreenL$, $\leq_\GreenR$ and $\leq_\GreenJ$)
which can be defined upon any semigroup, and which encapsulate the structure
of its maximal subgroups and principal left, right and two-sided ideals.
They are
powerful tools for understanding semigroups and monoids, and play a key role
in almost every aspect of
modern semigroup theory.
The relations $\leq_\GreenL$, $\leq_\GreenR$, $\GreenL$, $\GreenR$
and $\GreenH$ can be described in generality for the full matrix
semigroup over a semiring with identity, and hence present no particular
challenge in the tropical case; see \cite{K_tropicalgreen} or
Section~\ref{sec_greens} below for details. In \cite{K_tropicalgreen},
we initiated the study of Green's relations in tropical matrix semigroups,
by describing the remaining relations in the case of the $2 \times 2$
tropical matrix semigroup.
In \cite{K_tropd}, Hollings and the second author gave a complete description
of the $\GreenD$-relation in arbitrary finite dimensions, based on some deep
connections with the phenomenon of \textit{duality} between the row and
column space of a tropical matrix.

In the present paper, we turn our attention to the equivalence relation
$\GreenJ$ and pre-order $\leq_\GreenJ$ in the full tropical matrix semigroup
of arbitrary dimension. In the classical case of finite-dimensional matrices
over a field, it is well known that the relations $\GreenD$ and $\GreenJ$
coincide. Previous work of the authors showed that this correspondence
holds for $2 \times 2$ tropical matrices \cite{K_tropicalgreen}, but
Izhakian and Margolis \cite{IzhakianPrivComm} have shown that it does not
extend to higher dimensional tropical matrix semigroups over the tropical semiring with
$-\infty$; indeed, they have found an example
of a $\GreenJ$-class in the $3 \times 3$ tropical matrix semigroup which
contains infinitely many $\GreenD$-classes.

Our main technical result (Theorem~\ref{jcharrow}) gives a precise 
characterisation of the $\GreenJ$-order (and hence also of 
$\GreenJ$-equivalence) in terms of morphisms between certain tropical convex sets: 
specifically, $A \OrdJ B$ exactly if there exists a convex set $Y$ such 
that the row space of $B$ maps surjectively onto $Y$, and the row space of 
$A$ embeds injectively into $Y$. Using duality theorems, we also show
(Theorem~\ref{jnecessary}) that 
if $A \OrdJ B$ then the row space of $A$ admits an isometric (with respect 
to the Hilbert projective metric) embedding into the row space of $B$, but 
we show that in general it need \textit{not} admit a linear embedding. 
From these results, we are able to deduce that the semigroup of $n \times 
n$ matrices over the \textit{finitary} tropical semiring (without 
$-\infty$) does satisfy $\GreenD = \GreenJ$.

In the classical case of a finite-dimensional full matrix
semigroup over a field, it is well known that the $\GreenJ$-relation
(and the $\GreenD$-relation, with which it coincides)
encapsulates the concept of \textit{rank}, with the $\GreenJ$-order
corresponding to the obvious order on ranks. Thus, for the semigroup of
matrices over a more general ring or semiring, the $\GreenJ$-class of a
matrix may be thought of as a natural analogue of its rank. This idea
is rather different to traditional notions of rank since it is non-numerical,
taking values in a poset (the $\GreenJ$-order), rather than the natural
numbers.

For tropical matrices, several different (numerical) notions of rank have
been proposed and studied,
both separately and in relation to one another
(see for example \cite{Akian06,Akian09,Butkovic10,Develin05,Izhakian09,Shitov10}). Each of these clearly has merit for particular applications, but
overall we suggest that the proliferation of incompatible definitions is
evidence that the kind of information given by the ``rank'' of a classical
matrix cannot, in the tropical case, be encapsulated in a single natural
number. We believe that the $\GreenJ$-class of a matrix may serve as a
``general purpose'' analogue of rank for tropical mathematics, and partly
with this in mind, the final section of this paper discusses the
relationship between $\GreenJ$-class and some existing notions of rank.

\section{Preliminaries}\label{sec_prelim}

In this section we briefly recall the foundational definitions of tropical
algebra, and establish some elementary properties which will be required
later.

The {\em finitary tropical semiring} $\mathbb{FT}$ is the semiring (without additive identity)
consisting of the real numbers under the operations of addition and maximum.
We write $a \oplus b$ to denote the maximum of $a$ and $b$, and $a \otimes b$
or just $ab$ to denote the sum of $a$ and $b$. Note that both operations are
associative and commutative and that $\otimes$ distributes over $\oplus$.

The {\em tropical semiring} $\mathbb{T}$ is the finitary tropical semiring augmented with an
extra element $-\infty$ which acts as a zero for addition and an identity for maximum.
The {\em completed tropical semiring} $\overline{\mathbb{T}}$ is the tropical
semiring augmented
with an extra element $+\infty$, which acts as a zero for both maximum and addition,
save that $$(-\infty )(+\infty)=(+\infty )(-\infty)=-\infty.$$ Thus
$\mathbb{FT} \subseteq \mathbb{T} \subseteq \overline{\mathbb{T}}$ and
we call the elements of $\mathbb{FT}$ finite elements.

For any commutative semiring $S$, we denote
by $M_n(S)$ the set of all $n \times n$ matrices with entries drawn from
$S$. This has the structure of a semigroup, under the multiplication induced
from the semiring operations in the usual way.

We extend the usual order $\leq$ on $\mathbb{R}$ to a total order on $\mathbb{T}$ and $\overline{\mathbb{T}}$ by
setting $-\infty < x < +\infty$ for all $x \in \mathbb{R}$. Note that $a \oplus b = a$ exactly if $b \leq a$. The semirings $\mathbb{FT}$ and $\overline{\mathbb{T}}$ admit a natural order-reversing involution $x \mapsto -x$,
where of course $- (- \infty) = +\infty$ and $-(+\infty) = -\infty$.

For $S \in \{\mathbb{FT},\mathbb{T}, \overline{\mathbb{T}}\}$ we shall be interested in the space $S^n$ of \emph{affine tropical vectors}. We write $x_i$ for the $i$th component of a vector $x \in S^n$. We extend
$\oplus$ and $\leq$ to $S^n$ componentwise so that $(x\oplus y)_i = x_i \oplus y_i$ and $x \leq y$ exactly if $x_i \leq y_i$ for all $i$. We define a scaling action of $S$ on $S^n$ by
$$\lambda \otimes (x_1, \ldots, x_n) = (\lambda \otimes x_1, \ldots, \lambda \otimes x_n)$$
for each $\lambda \in S$ and each $x \in S^n$.
Similarly, for $S \in \{ \ft,\olt \}$ we extend the involution $x \mapsto -x$
on $S$ to $S^n$ by defining $(-x)_i = -(x_i)$. The scaling and $\oplus$
operations give $S^n$ the structure of an \textit{$S$-module} (sometimes
called an $S$-semimodule since the $\oplus$ operation does not admit
inverses).

From affine tropical $n$-space
we obtain projective tropical $(n-1)$-space (denoted $\mathbb{PFT}^{(n-1)}$,
$\mathbb{PT}^{(n-1)}$ or $\overline{\mathbb{PT}}^{(n-1)}$ as appropriate) by
identifying two vectors if one is a tropical multiple of the other by an
element of $\mathbb{FT}$.

An \textit{$S$-linear convex set} in $S^n$ is a subset closed under $\oplus$ and
scaling by elements of $S$, that is, an $S$-submodule of $S^n$. If $B \subseteq S^n$ then the ($S$-linear)
\textit{convex
hull} of $B$ is smallest convex set containing $B$, that is, the set of all
vectors in $S^n$ which can be written as tropical
linear combinations of finitely many vectors from $B$. Given two convex
sets $X \subseteq S^n$ and $Y \subseteq S^m$, we say that $f: X \rightarrow Y$
is a \emph{linear map} from $X$ to $Y$ if $f(x \oplus x') = f(x) \oplus f(x')$ and $f(\lambda \otimes x)= \lambda\otimes f(x)$ for all $x, x' \in X$ and all $\lambda \in S$.

Since each convex set $X \subseteq S^n$ is closed under scaling, it induces
a subset of the corresponding
projective space, termed the \emph{projectivisation} of $X$. Notice that
one convex set contains another exactly if there is a corresponding
containment of their projectivisations.

Given a
matrix $A \in M_n(S)$ we define the \textit{row space} of $A$, denoted
$R_S(A)$, to be the $S$-linear convex hull of the rows of $A$. Thus
$R_S(A) \subseteq S^n$. Similarly, we define the \textit{column space}
$C_S(A) \subseteq S^n$ to be the $S$-linear convex hull of the columns
of $A$. We shall also be interested in the projectivisation of $C_S(A)$,
which we call the projective column space of $A$ and denote $PC_S(A)$.
Dually, the projective row space $PR_S(A)$ is the projectivisation of
the row space of $A$.

We define a scalar product operation $\overline{\mathbb{T}}^n \times \overline{\mathbb{T}}^n \rightarrow \overline{\mathbb{T}}^n$ on affine tropical $n$-space by setting
$$\langle x \mid y \rangle = {\rm max}\{\lambda \in \overline{\mathbb{T}} : \lambda \otimes x \leq y\}.$$
This is a \textit{residual} operation in the sense of residuation
theory \cite{Blyth72}, and has been frequently employed in max-plus algebra.
Notice that $\langle x \mid y \rangle = +\infty$ if and only if for each
$i$ either $x_i=-\infty$ or $y_i=+\infty$. Thus $\langle x \mid x \rangle = +\infty$ if and only if
$x_i \in \{-\infty, +\infty\}$ for all $i$. It also follows that if $x, y \in \mathbb{T}$ with
$x \neq (-\infty, \ldots, -\infty)$ then $\langle x \mid y \rangle \in \mathbb{T}$. Similarly, we note
that $\langle x \mid y \rangle = -\infty$ if and only if there exists $j$ such that
either $x_j=+\infty \neq y_j$ or $y_j=-\infty \neq x_j$.  Thus if $x, y \in \mathbb{FT}^n$ then
$\langle x \mid y \rangle \in \mathbb{FT}$.

\begin{lemma}
\label{angleopp}
Let $x,y \in \overline{\mathbb{T}}^n$ with $x \neq y$. If $\langle x \mid y \rangle = +\infty$ then
$\langle y \mid x \rangle = -\infty$.
\end{lemma}

\begin{proof}
Since $x \neq y$ there exists $j$ such that $x_j \neq y_j$. Now, $\langle x \mid y \rangle = +\infty$
implies that for each $i$ either $x_i=-\infty$ or $y_i=+\infty$. Thus either $x_j=-\infty \neq y_j$
or $y_j=+\infty \neq x_j$ and hence, by the remarks preceding the lemma, we find that
$\langle y \mid x \rangle = -\infty$.
\end{proof}

We define a distance function on $\overline{\mathbb{T}}^n$ by $d_H(x,y)=0$ if $x$ is a finite scalar multiple of $y$ and
$$d_H(x,y) = -(\langle x \mid y \rangle \otimes \langle y \mid x \rangle)$$
otherwise. By Lemma \ref{angleopp}, it is easy to see that
$d_H(x,y) \neq -\infty$ for all $x,y \in \overline{\mathbb{T}}^n$. Thus
$d_H(x,y)=+\infty$ unless both $\langle x \mid y \rangle$ and $\langle y \mid x \rangle$ are finite. Moreover, if
$\langle x \mid y \rangle, \langle y \mid x \rangle \in \mathbb{FT}$ then
it is easy to check that $d_H(x,y) \geq 0$. It is also easily verified that
$d_H$ is invariant under scaling $x$ or $y$ by finite scalars
and hence is well-defined on $\overline{\mathbb{PT}}^{(n-1)}$,
$\mathbb{PT}^{(n-1)}$ and $\mathbb{PFT}^{(n-1)}$. For
$x,y \in \mathbb{FT}^n$ we see that $d_H(x,y) \in \mathbb{FT}$. In
fact, it can be shown that $d_H$ is a metric on $\mathbb{PFT}^{(n-1)}$
and an extended metric on $\overline{\mathbb{PT}}^{(n-1)}$ and
$\mathbb{PT}^{(n-1)}$, called
the \emph{(tropical) Hilbert projective metric} (see \cite[Proposition 1.6]{K_tropd}, for example).
In particular, $d_H$ induces obvious definitions of \textit{isometry} and
\textit{isometric embeddings} between subsets of tropical projective spaces.

Now let $S \in \{ \ft, \olt \}$ and let $A \in M_n(S)$. Following \cite{Cohen04} and \cite{Develin04} we
define a map $\theta_A: R_S(A) \rightarrow C_S(A)$ by
$\theta_A(x) = A\otimes (-x)^T$ for all $x \in R_S(A)$.
Dually, we define $\theta'_A : C_S(A) \rightarrow R_S(A)$
by $\theta'_A(x) = (-x)^T\otimes A$ for all $x \in C_S(A)$. We call $\theta_A$
and $\theta'_A$ the \emph{duality maps} for $A$. Notice that the duality
maps do not make sense over $S=\mathbb{T}$, as the involution $x \mapsto -x$
is not defined for $x = - \infty$. The following lemma recalls
some known properties of the duality maps which we shall need.

\begin{lemma}{\rm (Properties of the duality maps \cite{Cohen04,Develin04,K_tropd}).)}\\
\label{dualprop}
Let $S \in \{\mathbb{FT},\overline{\mathbb{T}}\}$ and let $A \in M_n(S)$.
\begin{itemize}
\item[(i)] $\theta_A$ and $\theta'_A$ are mutually inverse bijections between $R_S(A)$ and $C_S(A)$.
\item[(ii)] For all $x, y \in R_S(A)$, $x \leq y$ if and only $\theta_A(y) \leq \theta_A(x)$.\\
For all $x, y \in C_S(A)$, $x \leq y$ if and only $\theta'_A(y) \leq \theta'_A(x)$.\\
    We say that $\theta_A$ and $\theta'_A$ are order reversing.
\item[(iii)] For all $x\in R_S(A)$ and all $\lambda \in \mathbb{FT}$, $\theta_A(\lambda \otimes x) = -\lambda \otimes \theta_A(x)$.\\ For all $x\in C_S(A)$ and all $\lambda \in \mathbb{FT}$, $\theta'_A(\lambda \otimes x) = -\lambda \otimes \theta'_A(x)$.\\
We say that $\theta_A$ and $\theta'_A$ preserve scaling by finite scalars.
\end{itemize}
\end{lemma}

Of the properties in the lemma, part (i) is established for $S = \ft$
in \cite{Develin04} and for $S = \olt$ in \cite{Cohen04}. Part (ii) is
shown in \cite{Cohen04}. Part (iii) is proved in \cite{K_tropd}, which
also includes an expository account of the other two parts.

We now recall the
``metric duality theorem'' of \cite{K_tropd}.

\begin{theorem}{\rm (Metric duality theorem.)}\\
\label{dualmet}
Let $S \in \{\mathbb{FT},\overline{\mathbb{T}}\}$ and let $A \in M_n(S)$. Then the duality maps $\theta_A$ and $\theta'_A$ induce mutually inverse isometries (with respect to the Hilbert projective metric) between $PR_S(A)$ and $PC_S(A)$.
\end{theorem}

\section{Green's Relations}\label{sec_greens}

Green's relations are five equivalence relations and three pre-orders,
which can be defined on any semigroup, and which together describe the
(left, right and two-sided) principal ideal structure of the semigroup.
We give here brief definitions; for fuller discussion, proof of claimed
properties and equivalent formulations, we refer the reader to an
introductory text such as \cite{Howie95}.

Let $S$ be any semigroup. If $S$ is a monoid, we set $S^1 = S$, and otherwise
we denote by $S^1$ the monoid obtained by adjoining a new identity element
$1$ to $S$. We define a binary relation $\leq_\GreenR$ on $S$ by
$a \OrdR b$ if $a S^1 \subseteq b S^1$, that is, if either $a = b$
or there exists $q$ with $a = bq$. We define another relation $\GreenR$ by
$a \RelR b$ if and only if $a S^1 = b S^1$.

The relations $\leq_\GreenL$ and $\GreenL$ are the left-right duals of
$\leq_\GreenR$ and $\GreenR$, so $a \OrdL b$ if $S^1 a \subseteq S^1 b$,
and $a \RelL b$ if $S^1 a = S^1 b$. The relations $\leq_\GreenJ$ and $\GreenJ$
are two-sided analogues, so $a \OrdJ b$ if $S^1 a S^1 \subseteq S^1 b S^1$,
and $a \RelJ b$ if $S^1 a S^1 = S^1 b S^1$.
We also define a relation $\GreenH$ by $a \RelH b$ if $a \RelL b$ and
$a \RelR b$. Finally, the relation $\GreenD$ is defined by $a \RelD b$
if there exists an element $c \in S$ such that $a \RelR c$ and $c \RelL a$.

The relations $\GreenR$, $\GreenL$, $\GreenH$, $\GreenJ$ and $\GreenD$ are
equivalence relations; this is trivial in the first four cases, but requires
slightly more work in the case of $\GreenD$. The relations $\leq_\GreenR$,
$\leq_\GreenL$ and $\leq_\GreenJ$ are pre-orders (reflexive, transitive binary
relations) each of which induces a partial order on the equivalence classes
of the corresponding equivalence relation.

The study of Green's relations for the full tropical matrix semigroups was
begun (in the $2 \times 2$ case) by the authors \cite{K_tropicalgreen} and continued
in greater generality by Hollings and the second author \cite{K_tropd}. Some
key results of those papers are summarised in the following two theorems;
see \cite[Proposition 3.1]{K_tropd}, \cite[Theorem 5.1]{K_tropd}, \cite[Theorem 5.5]{K_tropd}
and \cite[Theorem 3.5]{K_tropd} for full details and proofs.

\begin{theorem}{\rm (Known characterisations of Green's Relations.)}\\
\label{greenchar}
Let $A, B \in M_n(S)$ for $S \in \{\mathbb{FT},\mathbb{T}, \overline{\mathbb{T}}\}$.
\begin{itemize}
\item[(i)] $A \OrdL B$ if and only if $R_S(A) \subseteq R_S(B)$;
\item[(ii)] $A \RelL B$ if and only if $R_S(A) = R_S(B)$;
\item[(iii)] $A \OrdR B$ if and only if $C_S(A) \subseteq C_S(B)$;
\item[(iv)] $A \RelR B$ if and only if $C_S(A) = C_S(B)$;
\item[(v)] $A \RelH B$ if and only if $R_S(A) = R_S(B)$ and $C_S(A) = C_S(B)$;
\item[(vi)] $A \RelD B$ if and only if  $C_S(A)$ and $C_S(B)$ are isomorphic as $S$-modules;
\item[(vii)] $A \RelD B$ if and only if  $R_S(A)$ and $R_S(B)$ are isomorphic as $S$-modules.
\end{itemize}
\end{theorem}

\begin{theorem}{\rm (Inheritance of $\mathcal{L}, \mathcal{R}$, $\mathcal{H}$ and $\mathcal{D}$.)}\\
\label{greeninherit}
Consider $M_n(\mathbb{FT})\subseteq M_n(\mathbb{T}) \subseteq M_n(\overline{\mathbb{T}})$. Each of Green's pre-orders $\leq_{\mathcal{L}}$, $\leq_{\mathcal{R}}$ and equivalence relations $\mathcal{L}, \mathcal{R}$, $\mathcal{H}$ and $\mathcal{D}$ in $M_n(\mathbb{FT})$ or $M_n(\mathbb{T})$ is the restriction of the corresponding relation in $M_n(\overline{\mathbb{T}})$.
\end{theorem}

The $\mathcal{J}$-relation and $\leq_{\mathcal{J}}$ pre-order for the semigroups $M_n(S)$ with $S \in \{\mathbb{FT},\mathbb{T}, \overline{\mathbb{T}}\}$ have so far remained rather mysterious. We shall give a characterisation of the $\mathcal{J}$-relation in these tropical matrix semigroups.

\section{The $\mathcal{J}$-relation is inherited}

In many applications one wishes to work with the tropical semiring
$\mathbb{T}$, but for theoretical purposes it is often nicer to work
over the finitary tropical semiring $\ft$ or the completed tropical semiring
$\overline{\mathbb{T}}$. The following result is an analogue for
the $\GreenJ$-order and $\GreenJ$-equivalence of Theorem \ref{greeninherit}
above, saying that these relations in a full matrix semigroup over $\trop$
and $\ft$ are inherited from the corresponding semigroup over $\olt$.
Hence, in order to understand these relations in all three cases,
it suffices to study them for $\olt$, and we shall for much of the remainder of the
paper work chiefly with $\olt$.

\begin{proposition}{\rm (Inheritance of $\mathcal{J}$.)}\\
\label{jinherit}
Consider $M_n(\mathbb{FT})\subseteq M_n(\mathbb{T}) \subseteq M_n(\overline{\mathbb{T}})$.
\begin{itemize}
\item[(i)] Let $A,B \in M_n(\mathbb{FT})$. Then
$$A \OrdJ B \textrm{ in } M_n(\mathbb{FT}) \textrm{ if and only if } A \OrdJ B \textrm{ in } M_n(\mathbb{T}).$$
\item[(ii)] Let $A,B \in M_n(\mathbb{T})$. Then
$$A \OrdJ B \textrm{ in } M_n(\mathbb{T}) \textrm{ if and only if } A \OrdJ B \textrm{ in } M_n(\overline{\mathbb{T}}).$$
\end{itemize}
\end{proposition}

\begin{proof}
(i) It is clear that if $A \OrdJ B$ in $M_n(\mathbb{FT})$  then $A \OrdJ B$ in $M_n(\mathbb{T})$. Suppose now that $A \OrdJ B$ in $M_n(\mathbb{T})$. Thus there exist $P,Q \in M_n(\mathbb{T})$ such that $A=PBQ$ giving
\begin{equation}\label{eq_dag}
A_{i,j} = \bigoplus_{k=1}^{n} \bigoplus_{l=1}^n P_{i,k}B_{k,l}Q_{l,j}
\end{equation}
for all $i, j \in \lbrace 1, \dots, n \rbrace$.
Since $P,B,Q$ each have finitely many entries we may choose $\delta \in \mathbb{R}$ such that:
\begin{itemize}
\item[(1)] $\delta \leq p+b+q-b'-p'$ for every pair of (necessarily finite) entries $b,b'$ in B and all finite entries $p,p'$ in $P$ and $q$ in $Q$;
\item[(2)] $\delta \leq p+b+q-b'-q'$ for every pair of (necessarily finite) entries $b,b'$ in B and all finite entries $p$ in $P$ and $q,q'$ in $Q$;
\item[(3)] $2\delta \leq p+b+q-b'$ for every pair of (necessarily finite) entries $b,b'$ in B and all finite entries $p$ in $P$ and $q$ in $Q$.
\end{itemize}
Now let $P', Q'$ be the matrices obtained from $P$ and $Q$ respectively by
replacing each $-\infty$ entry by $\delta$. Thus $P', Q' \in M_n(\mathbb{FT})$.
We shall show that $A=P'BQ'$.

Let $i, j \in \lbrace 1, \dots, n \rbrace$.
By \eqref{eq_dag} we may choose $k$ and $l$ such that $A_{i,j}=P_{i,k}B_{k,l}Q_{l,j} \geq P_{i,h}B_{h,m}Q_{m,j}$ for all $h$ and $m$. Since $A_{i,j}$ is finite it follows that $P_{i,k}$ and $Q_{l,j}$ are also finite. Thus $P_{i,k}'=P_{i,k}$ and $Q_{l,j}'=Q_{l,j}$, giving $A_{i,j}=P_{i,k}'B_{k,l}Q_{l,j}'$. It then suffices to show that $(P'BQ')_{i,j} = P_{i,k}'B_{k,l}Q_{l,j}'$. Now
$$(P'BQ')_{i,j} = \bigoplus_{h=1}^{n} \bigoplus_{m=1}^n P_{i,h}'B_{h,m}Q_{m,j}',$$
so it will suffice to show that $P_{i,k}'B_{k,l}Q_{l,j}' \geq P_{i,h}'B_{h,m}Q_{m,j}'$ for all $h$ and $m$. There are four cases to consider:
\begin{itemize}
\item[(a)] If $P_{i,h}, Q_{m,j} \in \mathbb{FT}$ then
$$P_{i,h}'B_{h,m}Q_{m,j}' = P_{i,h}B_{h,m}Q_{m,j} \leq P_{i,k}B_{k,l}Q_{l,j} = P_{i,k}'B_{k,l}Q_{l,j}'.$$
\item[(b)] If $P_{i,h}\in \mathbb{FT}$ and $Q_{m,j}=-\infty$ then
\begin{eqnarray*}
P_{i,h}'B_{h,m}Q_{m,j}' &=& P_{i,h}+B_{h,m}+\delta\\
 &\leq& P_{i,h}+B_{h,m} + P_{i,k} +B_{k,l}+Q_{l,j}-B_{h,m}-P_{i,h}\\
 &\leq& P_{i,k} +B_{k,l}+Q_{l,j}= P_{i,k}B_{k,l}Q_{l,j} = P_{i,k}'B_{k,l}Q_{l,j}' ,
\end{eqnarray*}
by (1).
\item[(c)] If $P_{i,h}=-\infty$ and $Q_{m,j}\in \mathbb{FT}$ then we may
apply an argument dual to that in case (b), using condition (2) in place
of condition (1).

\item[(d)] If $P_{i,h}=-\infty$ and $Q_{m,j}=-\infty$ then
\begin{eqnarray*}
P_{i,h}'B_{h,m}Q_{m,j}' &=& \delta+B_{h,m}+\delta\\
 &\leq& P_{i,k}+B_{k,l} + Q_{l,j} -B_{h,m}+B_{h,m}\\
 &\leq& P_{i,k} +B_{k,l}+Q_{l,j}= P_{i,k}B_{k,l}Q_{l,j} = P_{i,k}'B_{k,l}Q_{l,j}' ,
\end{eqnarray*}
by (3).
\end{itemize}
Thus we see that $(P'BQ')_{i,j}=A_{i,j}$ for all $i$ and $j$.

(ii) It is clear that if $A \OrdJ B$ in $M_n(\mathbb{T})$
then $A \OrdJ B$ in $M_n(\overline{\mathbb{T}})$. Suppose
now that $A \OrdJ B$ in $M_n(\overline{\mathbb{T}})$. Thus
there exist $P,Q \in M_n(\overline{\mathbb{T}})$ such that $A=PBQ$. Let
$P', Q'$ be the matrices obtained from $P$ and $Q$ respectively by
replacing each $+\infty$ entry by $0$. Then it is straightforward to
check, by an argument similar to the above, that $A=P'BQ'$.
\end{proof}

\section{Characterising the $\mathcal{J}$-order}

In this section we shall give an exact characterisation of the
$\mathcal{J}$-order, and hence also of the $\mathcal{J}$-relation,
in terms of linear morphisms between column spaces (or dually, row spaces). As
discussed in the previous section, we restrict our attention to the
semirings $\mathbb{FT}$ and $\overline{\mathbb{T}}$, enabling us to make
use of the duality maps and the Metric Duality Theorem
(Theorem~\ref{dualmet} above).
Since the $\mathcal{J}$-relation in $M_n(\mathbb{T})$ is the restriction
of the corresponding relation in $M_n(\overline{\mathbb{T}})$, this also
gives a complete characterisation of $\mathcal{J}$.

We first recall the following result from \cite{K_tropd}.
\begin{theorem}
\label{useful}
Let $A, B \in M_n(S)$ for $S \in \{\mathbb{FT},\overline{\mathbb{T}}\}$. Then the following are equivalent:
\begin{itemize}
\item[(i)] $R_S(A) \subseteq R_S(B)$;
\item[(ii)] there is a linear morphism from $C_S(B)$ to $C_S(A)$ taking the $i$th column of $B$ to the $i$th column of $A$ for all $i$;
\item[(iii)] there is a surjective linear morphism from $C_S(B)$ to $C_S(A)$ taking the $i$th column of $B$ to the $i$th column of $A$ for all $i$.
\end{itemize}
\end{theorem}

We remark that Theorem~\ref{useful} has a left-right dual, obtained by
swapping rows with columns and row spaces with column spaces throughout
the statement.
Theorem~\ref{useful} describes a duality between embeddings of row spaces and surjections
of column spaces, which also has an algebraic manifestation:

\begin{theorem}
\label{embedsurject}
Let $S \in \{\mathbb{FT},\overline{\mathbb{T}}\}$ and $A, B \in M_n(S)$.
Then the following are equivalent:
\begin{itemize}
\item[(i)] $C_S(B)$ surjects linearly onto $C_S(A)$;
\item[(ii)] $R_S(A)$ embeds linearly into $R_S(B)$;
\item[(iii)] there exists $C \in M_n(S)$ with $A \RelR C \OrdL B$.
\end{itemize}
\end{theorem}

\begin{proof}
We prove first that (i) implies (iii). Suppose that
$f : C_S(B) \to C_S(A)$
is a linear surjection. Let $C$ be the matrix obtained by applying $f$ to each
column of $B$. Then clearly $C_S(C) = C_S(A)$ so $C \RelR A$ by
Theorem~\ref{greenchar}(iv). Moreover, by
Theorem~\ref{useful} and the definition of $C$ we have $R_S(C) \subseteq R_S(B)$, so that by
Theorem~\ref{greenchar}(i) we have $C \OrdL B$.

Next we show that (iii) implies (ii). Since $A \RelR C$, in particular
$A \RelD C$, so Theorem~\ref{greenchar}(vii)
tells us that there is a linear isomorphism from $R_S(A)$ to $R_S(C)$.
Also, since $C \OrdL B$,
Theorem~\ref{greenchar}(i) gives that $R_S(C)$ is contained in $R_S(B)$.
Thus, the isomorphism gives a linear embedding of $R_S(A)$ into $R_S(B)$.

Finally, suppose (ii) holds, say
$f: R_S(A) \to R_S(B)$ is a linear embedding. Let $A'$ be
obtained from $A$ by applying $f$ to each row of $A$. Then
$R_S(A)$ is linearly isomorphic to $R_S(A')$, which is contained in
$R_S(B)$. By
Theorem \ref{useful}, it follows from the latter that there is a
linear surjection from
$C_S(B)$ onto $C_S(A')$. Moreover, since $R_S(A)$ and $R_S(A')$ are isomorphic
as $S$-modules, Theorem \ref{greenchar} parts (vi) and (vii) give that
$C_S(A)$ and $C_S(A')$ are isomorphic as $S$-modules. Composing gives a
linear surjection from $C_S(B)$ onto $C_S(A)$.
\end{proof}

We remark that the equivalence of conditions (i) and (ii) in
Theorem~\ref{embedsurject} is a manifestation of a more general
abstract 
categorical duality in residuation theory (see for example \cite{Gaubert96b}).
Again, the theorem has a left-right dual, obtained by
interchanging row spaces with column spaces, $\GreenL$ with $\GreenR$, and
$\leq_\GreenL$ with $\leq_\GreenR$. Theorem~\ref{embedsurject} and its dual lead easily to the main result of this
section.

\begin{theorem}{\rm (Linear characterisation of the $\mathcal{J}$-order.)} \\
\label{jcharrow}
Let $A, B \in M_n(S)$ for $S \in \{\mathbb{FT},\overline{\mathbb{T}}\}$. The following are equivalent.
\begin{itemize}
\item[(i)] $A \OrdJ B$;
\item[(ii)] there exists a convex set $Y \subseteq S^n$ such that $R_S(A)$ embeds linearly into $Y$ and $R_S(B)$ surjects linearly onto $Y$.
\item[(iii)] there exists a convex set $Y \subseteq S^n$ such that $R_S(A) \subseteq Y$ and $R_S(B)$ surjects linearly onto $Y$.
\item[(iv)] there exists a convex set $Y \subseteq S^n$ such that $C_S(A)$ embeds linearly into $Y$ and $C_S(B)$ surjects linearly onto $Y$.
\item[(v)] there exists a convex set $Y \subseteq S^n$ such that $C_S(A) \subseteq Y$ and  $C_S(B)$ surjects linearly onto $Y$.
\end{itemize}
\end{theorem}

\begin{proof}
We show the equivalence of (i), (ii) and (iii), since the equivalence of
(i), (iv) and (v) is dual. Suppose that $A \OrdJ B$,
say $A=PBQ$ for some $P,Q \in M_n(S)$. Thus $A \OrdL BQ$ and
$BQ \OrdR B$ so that $R_S(A) \subseteq R_S(BQ)$ and
$C_S(BQ) \subseteq C_S(B)$ by Theorem~\ref{greenchar}. By the dual to
Theorem \ref{useful}, there exists a surjective linear map from $R_S(B)$
onto $R_S(BQ)$. Thus, setting $Y=R_S(BQ)$ yields (iii).

That (iii) implies (ii) is trivial, so it remains only to show that (ii)
implies (i). Let $Y$ be a convex set with the given properties. Since $Y$
is a morphic image of $R_S(B)$, it is generated by the images of the rows
of $B$. Thus, we may suppose
that $Y$ is the row space of some matrix, say $F \in M_n(S)$. Now by the dual
to Theorem~\ref{embedsurject},
there is a matrix $C$ with $F \RelL C \OrdR B$, and by
Theorem~\ref{embedsurject}, there is a matrix $D$ with $A \RelR D \OrdL F$.
Thus, $A \OrdJ B$, so (i) holds.
\end{proof}

In \cite{K_tropicalgreen} we saw that the $\mathcal{J}$-relation on the
semigroup of $2 \times 2$ tropical matrices was characterised by the notion
of mutual isometric embedding of projective column (dually, row) spaces. We
show now that isometric embedding of projective column (dually, row) spaces
is a necessary condition of the $\mathcal{J}$-order, although we shall
see later (Section~\ref{sec_embednot} below) that it is not a sufficient
condition. The proof is based on the Metric Duality Theorem from \cite{K_tropd}
(Theorem~\ref{dualmet} above).

\begin{theorem}
\label{jnecessary}
Let $A, B \in M_n(S)$ for $S \in \{\mathbb{FT},\overline{\mathbb{T}}\}$. If $A \leq_{\mathcal{J}} B$ then
\begin{itemize}
\item[(i)] $PC_S(A)$ embeds isometrically into $PC_S(B)$;
\item[(ii)] $PR_S(A)$ embeds isometrically into $PR_S(B)$.
\end{itemize}
\end{theorem}

\begin{proof}
We show that (i) holds, with (ii) being dual.
Since $A \leq_{\mathcal{J}} B$ there exist $P, Q \in M_n(S)$ such that $A=PBQ$.
Now $A \leq_{\mathcal{R}} PB \leq_{\mathcal{L}} B$ so by
Theorem~\ref{greenchar} we have
$C_S(A) \subseteq C_S(PB)$ and $R_S(PB) \subseteq R_S(B)$. It follows that
$PC_S(A) \subseteq PC_S(PB)$ and $PR_S(PB) \subseteq PR_S(B)$. Now by
Theorem~\ref{dualmet}, $PC_S(PB)$ is isometric to $PR_S(PB)$, and
$PR_S(B)$ is isometric to $PC_S(B)$. By composing inclusions and
isometries in the appropriate order, we obtain an isometric embedding of
$PC_S(A)$ into $PC_S(B)$.
\end{proof}

We remark that there is no reason to believe that the isometric embedding
constructed in the proof of Theorem~\ref{jnecessary} is a linear morphism. The
inclusions are of course linear, and the isometries are ``anti-isomorphisms''
in the sense of \cite{K_tropd}, but this is not sufficient to ensure that
the composition is linear. The distinction here ultimately stems from the
distinction between a meet-semilattice morphism (which by definition preserves
greatest lower bounds) and an order-morphism of meet-semilattices (which
need not).

Theorem~\ref{embedsurject} (and also Theorem~\ref{jcharrow}) implies that
a \textit{linear} embedding of row spaces is a sufficient condition for
two matrices to be related in the $\GreenJ$-order, while Theorem~\ref{jnecessary} says that an
\textit{isometric} embedding of (projective) row spaces is a necessary
condition for the same property. It is very natural to ask, then, whether
the former condition is necessary, or the latter condition is sufficient. In
Section~\ref{sec_embednot} we shall give examples to show that, in general,
neither is the case.

\section{$\GreenD = \GreenJ$ in the finitary case}

One of the most fundamental structural questions about any semigroup is
whether the
relations $\GreenD$ and $\GreenJ$ coincide. These relations are always equal
in finite semigroups (more generally, in compact topological semigroups),
but differ in many important infinite semigroups. Semigroups in which
$\GreenD = \GreenJ$ have an ideal structure which is considerably easier
to analyse.

The full matrix semigroup $M_n(K)$ of matrices over a field $K$ is a
well-known example of
an (infinite, provided $K$ is infinite) semigroup in which
$\mathcal{D}=\mathcal{J}$. In \cite{K_tropicalgreen} we showed that
$M_2(\trop)$ also has $\GreenD = \GreenJ$, but Izhakian and Margolis
\cite{IzhakianPrivComm} have recently produced examples to show that
$\GreenD \neq \GreenJ$ in $M_3(\trop)$; it follows easily that $\GreenD \neq \GreenJ$
in $M_n(\trop)$ for all $n \geq 3$, and hence using
Proposition~\ref{jinherit} and Theorem~\ref{greeninherit} also in
$M_n(\olt)$ for $n \geq 3$. In contrast, in this section we shall show that the finitary
tropical matrix semigroup $M_n(\ft)$ satisfies $\GreenD = \GreenJ$ for
all $n$.

Our proof makes use of some topology.
For convenience, we identify $\pft^{n-1}$ with $\mathbb{R}^{n-1}$
via the correspondence
\begin{equation}
\label{project}
[(x_1, \dots, x_n)] \mapsto (x_1 - x_n, x_2 - x_n, \dots, x_{n-1} - x_n).
\end{equation}
With this identification, the Hilbert projective metric on $\pft^{n-1}$ is Lipschitz equivalent to
the standard Euclidean metric (which we will denote $d_E$) on $\mathbb{R}^{n-1}$. Indeed, it is an
easy exercise to verify that for any points $x, y \in \mathbb{R}^{n-1} = \pft^{n-1}$ we have
$$d_H(x,y) \leq \sqrt{2} d_E(x,y) \textrm{ and } d_E(x,y) \leq \sqrt{n-1} d_H(x,y).$$
In particular, the two metrics induce the same topology, so we may speak
without ambiguity of a sequence of points converging.

\begin{theorem}{\rm ($\GreenD = \GreenJ$ for finitary tropical matrices.)}\\
\label{finiteDJ}
Let $A,B \in M_n(\mathbb{FT})$. Then $A \RelJ B$ if and only if $A \RelD B$.
\end{theorem}

\begin{proof}
Clearly, if $A$ and $B$ are $\mathcal{D}$-related then they are also
$\mathcal{J}$-related. Now suppose for a contradiction that $A \RelJ B$,
but $A$ is not $\mathcal{D}$-related to $B$. We claim first that there is an
isometric (with respect to the Hilbert projective metric) self-embedding
$f: PR_\ft(A) \rightarrow PR_\ft(A)$ which is not an isometry
(that is, which is not surjective).

To construct such such a map, we proceed much as in the proof of
Theorem~\ref{jnecessary}.
Since $A \leq_{\mathcal{J}} B$ we may write $A = PBQ$ for some $P, Q \in M_n(\ft)$.
Letting $X = BQ$, we have $A \OrdL X \OrdR B$, so
by Theorem~\ref{greenchar},
$$PR_\ft(A) \subseteq PR_\ft(X) \text{ and } PC_\ft(X) \subseteq PC_\ft(B).$$
Now by
Theorem~\ref{dualmet}, there is an isometry from $PR_\ft(X)$ to $PC_\ft(X)$
and an isometry from $PC_\ft(B)$ to $PR_\ft(B)$. Composing these inclusions and
isometries in the appropriate order, we obtain an isometric embedding
$$h : PR_\ft(A) \rightarrow PR_\ft(B).$$
We claim that $h$ is not surjective. Indeed, for $h$ to be surjective we
would clearly have to have $PC_\ft(X)=PC_\ft(B)$ and $PR_\ft(A)=PR_\ft(X)$,
which by Theorem~\ref{greenchar} would yield $A \RelL X \RelR B$
and hence $A \RelD B$, giving a contradiction.

Dually, since $B \leq_{\mathcal{J}} A$, we may construct a
non-surjective isometric embedding $g: PR_\ft(B) \rightarrow PR_\ft(A)$.
Now let $f : PR_\ft(A) \rightarrow PR_\ft(A)$ be the non-surjective isometric
self-embedding given by the composition $g \circ h$.

By \cite[Proposition~2.6]{Joswig05}, the projective row space $PR_\ft(A)$
is a compact subset of $\mathbb{PFT}^{n-1} = \mathbb{R}^{n-1}$.
We will use this fact to
deduce the desired contradiction. Let $X_0 = PR_\ft(A)$. Since $f$ is not a surjection
and has closed
image, we may choose an $x_0 \in X_0$ and $\varepsilon>0$ such that
$x_0 \notin f(X_0)$ and $d_H(x_0,y) \geq \varepsilon$ for all
$y \in f(X_0)$. Now we define a sequence of points by
$x_i=f^i(x_0)$, and a decreasing sequence of sets by
$X_i = f^i(X_0)$.

Notice that, since $f^i$ is an isometric embedding, it follows from the
properties of $x_0$ and $X_0$ that each $x_i$ lies in $X_i$
but satisfies $d_H(x_i, y) \geq \varepsilon$ for all $y \in X_{i+1}$. In
particular, whenever $j > i$ we have $x_j \in X_j \subseteq X_{i+1}$ so
that $d_H(x_i, x_j) \geq \varepsilon$. It follows that the the $x_i$'s cannot
contain a convergent subsequence, which contradicts the fact that they are
contained in the compact set $X_0 = PR_\ft(A)$.
\end{proof}

\section{Embeddings do not characterise the $\GreenJ$-order.}\label{sec_embednot}

We have already seen that, by Theorem~\ref{embedsurject},
a linear embedding of row spaces is a sufficient condition for
two matrices to be related in the $\GreenJ$-order, while, by Theorem~\ref{jnecessary}, an
isometric embedding of (projective) row spaces is a necessary
condition for the same property. It is very natural to ask, then, whether an
exact characterisation of the $\GreenJ$-order can be obtained in terms of
(linear or isometric) embeddings alone. In this section we answer this
question in the negative, by giving examples to show that linear embedding of
row spaces is not a necessary condition for two matrices to be related
in the $\GreenJ$-order, while isometric embedding is not a sufficient
condition. By Theorem~\ref{embedsurject} and the fact that the
$\GreenJ$-order is invariant under taking matrix transposes,
this also suffices to exclude linear or isometric embeddings of column
spaces, or linear surjections of row or column spaces as exact
characterisations of the $\GreenJ$-order.

To show that our examples have the claimed properties, we shall need
some more concepts and terminology.
Let $S \in \{\mathbb{FT}, \mathbb{T}, \overline{\mathbb{T}}\}$ and let $X$
be a finitely generated convex set in $S^n$. A set $\{x_1, \ldots, x_k\} \subseteq X$
is called a {\em weak basis} of $X$ if it is a generating set for $X$ minimal with respect to inclusion. It is known that every finitely generated convex set admits a weak basis, which is unique up to permutation
and scaling (see \cite[Theorem 1 and Corollary 3.6]{Wagneur91} for the case $S\in \{\trop, \olt\}$ and \cite{CuninghameGreen04} for the case $S=\ft$). In particular, any two weak bases have the same cardinality,
in view of which  we may define the \textit{generator dimension} of a finitely
generated convex set $X$ to be the cardinality of a weak basis for $X$, or
equivalently, the minimum cardinality of a generating set for $X$. Generator
dimension is closely related to the notion which was called
\textit{linear independence} in \cite[Chapter~16]{CuninghameGreen79}.

Note
that generator dimension is \textit{not} well-behaved with respect to
inclusion: a convex set of generator dimension $k$ may contain a convex set of
generator dimension strictly greater than $k$. In particular, for $n \geq 3$ the
generator dimension of a finitely generated convex set in $S^n$ can exceed
$n$.

However, it is easily seen that generator dimension is well-behaved with respect
to linear surjections: if $f : X \to Y$ is a linear surjection then the generator
dimension of $X$ is at least that of $Y$. Indeed, if $B$ is a weak
basis for $X$ then $f(B)$ is a generating set for $Y$, and so $Y$ has
generator dimension at most $|f(B)| \leq |B|$. In particular, generator dimension is an
isomorphism invariant.

\begin{example}\label{ex_embednotnec}
(Linear embedding is not necessary for the $\GreenJ$-order.) \\
Consider the matrices
$$A=\left(\begin{array}{c c c c}
0& 1& 2 &3\\
0& -1&-2&-3\\
0&0&0&0\\
0&0&0&0
\end{array}\right) \textrm{ and } B=\left(\begin{array}{c c c c}
0&0& 1 & 3\\
0&0& -2 & -3\\
0&0&0&0\\
0&0&0&0
\end{array}\right).$$
in $M_4(\ft)$. Then $A \leq_{\mathcal{J}} B$; indeed $A \OrdR B$
since $A=BX$ where
$$X=\left(\begin{array}{c c c c}
0& -1 & -2 &-3\\
0 & -1 & -2 &-3\\
-1 & 0 &0&-1\\
-3&-3&-1&0
\end{array}\right),
$$
for example. Thus $C_{\ft}(A) \subseteq C_{\ft}(B)$ by Theorem \ref{greenchar}(iii). Since every element of
$C_{\ft}(A)$ and $C_{\ft}(B)$ has the form $(a,b,c,c)^T$, the elements of the corresponding
projective column spaces all have the form $(a-c,b-c,0)^T$ (where, as before, we identify $\pft^3$ with $\mathbb{R}^3$ via the map given in \eqref{project}). Since the third co-ordinate is fixed we may therefore draw our projective column spaces in 2-dimensions, as in Figure \ref{zigzag} below.

It is easy to check that $C_{\ft}(A)$ has generator dimension $4$, while
$C_{\ft}(B)$ has generator dimension $3$ (the points labelled in Figure \ref{zigzag} are the projectivisations of a weak basis for $C_{\ft}(A)$ and
$C_{\ft}(B)$ respectively).
It follows by the above discussion that there cannot be a surjective linear
morphism from $C_{\ft}(B)$ onto $C_{\ft}(A)$, and hence by
Theorem~\ref{embedsurject}, $R_{\ft}(A)$ does not embed in $R_{\ft}(B)$.

\ifpics{

\begin{figure}[h]
\caption{The projective column spaces of the matrices $A$ and $B$ from Example~\ref{ex_embednotnec}.}
\begin{pspicture}(7,4)
\psframe[fillstyle=solid, fillcolor=lightgray, linecolor=lightgray](1,2)(3,3)
\psframe[fillstyle=solid, fillcolor=lightgray, linecolor=lightgray](2,3)(3,1)
\qline(0,3)(3,3)
\qline(1,2)(2,2)
\qline(2,1)(3,1)
\qline(1,3)(1,2)
\qline(2,2)(2,1)
\qline(3,3)(3,0)
\rput(2.1,2.3){$PC_{\ft}(A)$}
\rput(-0.6,3){$(0, 0)$}
\rput(0.25,2){$(1, -1)$}
\rput(1.25,1){$(2, -2)$}
\rput(2.25,0){$(3, -3)$}
\psframe[fillstyle=solid, fillcolor=lightgray, linecolor=lightgray](6,3)(8,1)
\qline(5,3)(8,3)
\qline(6,1)(8,1)
\qline(6,3)(6,1)
\qline(8,3)(8,0)
\rput(7.1,2){$PC_{\ft}(B)$}
\rput(4.4,3){$(0, 0)$}
\rput(5.25,1){$(1, -2)$}
\rput(7.25,0){$(3, -3)$}
\psdots*(0,3)(1,2)(2,1)(3,0)(5,3)(6,1)(8,0)
\end{pspicture}
\label{zigzag}
\end{figure}
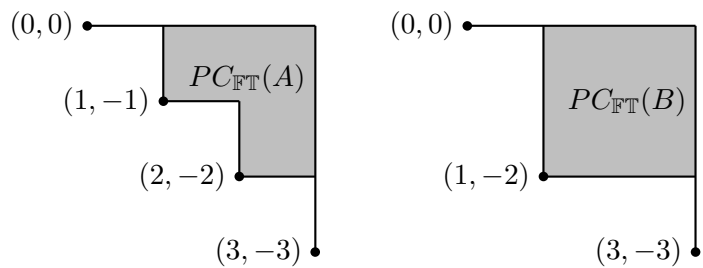
}

It is an easy exercise to extend the dimension of this example, so as
to show that linear embedding of row spaces is not necessary for the
$\GreenJ$-order in $M_n(\ft)$ for all $n \geq 4$ and hence by Proposition~\ref{jinherit}
also in $M_n(\trop)$ and $M_n(\olt)$ for $n \geq 4$. Using the results of \cite{K_tropicalgreen} it can be shown that in $M_2(\trop)$ (and hence also in $M_2(\ft)$, by Proposition~\ref{jinherit}), linear embedding of row spaces is an exact characterisation of the $\GreenJ$-order. The three dimensional case, that is, whether linear embedding of row spaces is an exact characterisation
of the $\GreenJ$-order in $M_3(\ft)$, $M_3(\trop)$ and/or $M_3(\olt)$ remains open.
\end{example}

Example~\ref{ex_embednotnec} shows that linear embedding does not (even in the finitary
case) characterise the $\GreenJ$-order, but it does not rule out the
possibility that \textit{mutual} linear embedding characterises the
$\GreenJ$-equivalence. Indeed, in the finitary case,
Theorem~\ref{greenchar}(vi) and Theorem~\ref{finiteDJ} together imply
that $\GreenJ$-related matrices have linearly isomorphic column spaces,
and hence in particular, mutually embedding column spaces. The following
example shows that this characterisation does not extend to the case of
$\trop$.

\begin{example}\label{ex_mutualembednotnec}
(Mutual linear embedding is not necessary for $\GreenJ$.) \\

Consider the matrices
$$A=\left(\begin{array}{c c c c}
-\infty& 0& 1 & 1\\
-\infty& -\infty&1& 1 \\
0&0&0&0\\
-\infty&-\infty&-\infty&-\infty
\end{array}\right), B=\left(\begin{array}{c c c c}
-\infty&0& 1 & 1\\
-\infty&-\infty& 1&0\\
0&0&0&0\\
-\infty&-\infty&-\infty&-\infty
\end{array}\right)$$
in $M_4(\trop)$. We first claim that $A \mathcal{J} B$. Let $\mu$ denote the linear
embedding $\mu: \mathbb{T}^4 \rightarrow \mathbb{T}^4$ given by
$\mu: (x,y,z,t) \mapsto (x,y,z+1,t)$. It is immediate that the restriction
of $\mu$ to $C_{\trop}(A)$ gives an embedding $\mu :C_{\trop}(A) \rightarrow C_{\trop}(B)$ and
straightforward to check that the restriction of $\mu$ to $C_{\trop}(B)$ gives an
embedding $\mu :C_{\trop}(B) \rightarrow C_{\trop}(A)$. Thus $\mu$ gives mutual linear
embeddings of the column spaces. Hence, by the dual to Theorem \ref{embedsurject}, we
deduce that $A \mathcal{J} B$. It is also easy to check that 
$C_{\trop}(A)$ has generator dimension $3$ while
$C_{\trop}(B)$ has generator dimension $4$. So by the
same argument as in Example~\ref{ex_embednotnec}, $C_{\trop}(A)$ cannot
surject onto $C_{\trop}(B)$, and hence $R_{\trop}(B)$ cannot embed in $R_{\trop}(A)$.

Every non-zero element of $C_{\trop}(A)$ and $C_{\trop}(B)$ has the form
$(a,b,c,-\infty)^T$, where $c \neq -\infty$. Thus, we may identity the
elements of the projectivisations with elements of the form
$(a-c,b-c,0, -\infty)^T$. Since the third and fourth co-ordinates are
fixed we may therefore draw our projective column spaces in two dimensions,
as in Figure \ref{inf} below. The projective row space of $A$ can be drawn
similarly, but that of $B$ is harder to illustrate in two dimensions.

\ifpics{

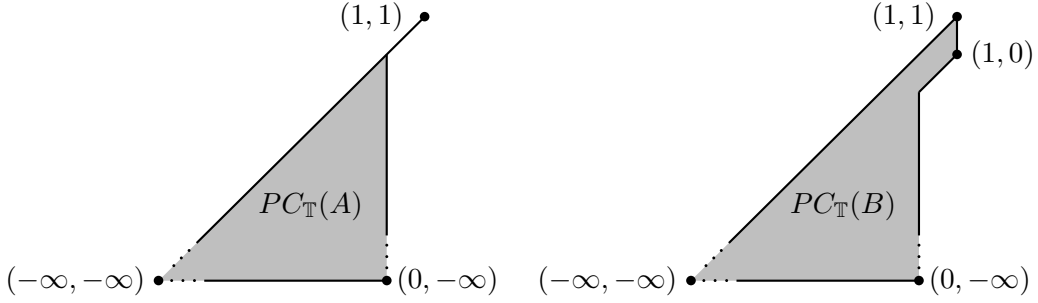
\begin{figure}[h]
\caption{The projective column spaces of the matrices $A$ and $B$ from Example~\ref{ex_mutualembednotnec}.}
\begin{pspicture}(11.2,4.5)
\pspolygon*[linecolor=lightgray](1,0)(4,0)(4,3)
\rput(3,1){$PC_{\trop}(A)$}
\rput(-0.1,0){$(-\infty, -\infty)$}
\rput(4.8,0){$(0, -\infty)$}
\rput(3.8,3.5){$(1, 1)$}
\qline(1.5,0.5)(4.5,3.5)
\qline(1.6,0)(4,0)
\qline(4,3)(4,0.6)
\pspolygon*[linecolor=lightgray](8,0)(11,0)(11,3)
\pspolygon*[linecolor=lightgray](11,3)(11.5,3.5)(11.5,3)(11,2.5)
\rput(10,1){$PC_{\trop}(B)$}
\rput(6.9,0){$(-\infty, -\infty)$}
\rput(11.8,0){$(0, -\infty)$}
\rput(10.8,3.5){$(1, 1)$}
\rput(12.1,3){$(1, 0)$}
\qline(8.5,0.5)(11.5,3.5)
\qline(11.5,3.5)(11.5,3)
\qline(11.5,3)(11,2.5)
\qline(11,2.5)(11,0.6)
\qline(8.6,0)(11,0)
\psdots*(1,0)(4,0)(4.5,3.5)(8,0)(11,0)(11.5,3.5)(11.5,3)
\psdots*[dotsize=1.1pt](1.3,0.3)(1.4,0.4)(1.2,0.2)(1.32,0)(1.47,0)(1.17,0)(4,0.32)(4,0.47)(4,0.17)
(8.3,0.3)(8.4,0.4)(8.2,0.2)(8.32,0)(8.47,0)(8.17,0)(11,0.32)(11,0.47)(11,0.17)
\end{pspicture}
\label{inf}
\end{figure}

}

Again, it is straightforward to extend this example to higher dimensions,
showing that mutual linear embedding of row spaces is not necessary for
two matrices to be $\GreenJ$-related in $M_n(\trop)$ for $n \geq 4$, and hence by
Proposition~\ref{jinherit} also in $M_n(\olt)$, for $n \geq 4$. In two
dimensions, we know from \cite{K_tropicalgreen} that $\GreenD = \GreenJ$
in $M_2(\trop)$, so by the same argument as in the finitary case we see that
mutual linear embedding exactly characterises $\GreenJ$. The three
dimensional case, that is, whether mutual linear embedding of row spaces
is necessary for $\GreenJ$-equivalence in $M_3(\trop)$ and/or $M_3(\olt)$,
remains open.
\end{example}

\begin{example}\label{ex_isometrynotsuff}
(Isometry is not sufficient for $\GreenJ$-equivalence.) \\

Consider the matrix
$$A=\left(\begin{array}{c c c c}
0& 0& 0\\
1& 5&0\\
3&2&0\\
\end{array}\right)$$
in $M_3(\mathbb{FT})$. By Theorem~\ref{dualmet} we have that $PR_{\ft}(A)$ is isometric to $PC_{\ft}(A)$, or equivalently, $PR_{\ft}(A)$ is isometric to $PR_{\ft}(A^T)$. As usual, we identify $\pft^2$ with $\mathbb{R}^2$ via the map given in \eqref{project}. Figure \ref{hex} illustrates these isometric row spaces. We claim that $A$ is not $\mathcal{J}$-related to $A^T$.

Since $A$ has only finite entries, by Theorem~\ref{finiteDJ} $A \RelJ A^T$
if and only if $A \RelD A^T$, which by Theorem~\ref{greenchar}(vii) holds
if and only if $R_\ft(A)$ and $R_\ft(A^T)$ are linearly isomorphic. Thus,
it will suffice to show that these spaces are not linearly isomorphic.

Suppose for a contradiction that $f : R_{\ft}(A) \rightarrow R_{\ft}(A^T)$ is an isomorphism
of $S$-modules. Then $f$ induces an isometry $\hat{f}$ between the projective
row space $PR_{\ft}(A)$ and the projective row space $PR_{\ft}(A^T)$ mapping each
$x \in PR_{\ft}(A)$ to the projectivisation of $f(x)$ in $PR_{\ft}(A^T)$.
Since $f$ is an isomorphism, it maps weak bases to weak bases. We fix a weak
basis for $R_{\ft}(A)$ and let $\{a,b,c\}$ denote
the projectivisation of this basis. Similarly, fix any weak basis for
$R_{\ft}(A^T)$ and let $\{x,y,z\}$ denote the projectivisation of this basis.
It then follows that
$$\{ d_H(a,b), d_H(a,c), d_H(b,c)\} = \{ d_H(x,y), d_H(x,z), d_H(y,z)\}.$$
However, since the rows of $A$ form a weak basis of
$R_{\ft}(A)$ and the columns of $A$ form a weak basis of $R_{\ft}(A^T)$, we
find that
$$\{ d_H(a,b), d_H(a,c), d_H(b,c)\} = \{1,4,5\}$$ whilst
$$\{ d_H(x,y), d_H(x,z), d_H(y,z)\}=\{2,3,5\},$$
contradicting the existence of an isomorphism between $R_{\ft}(A)$ and $R_{\ft}(A^T)$.

\ifpics{

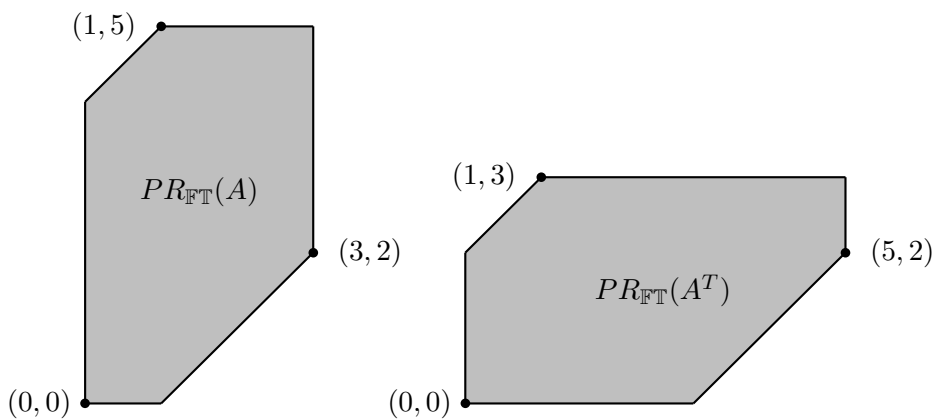
\begin{figure}[h]
\caption{The projective row spaces of the matrices $A$ and $A^T$ from Example~\ref{ex_isometrynotsuff}.}
\label{hex}
\begin{pspicture}(10,6)
\pspolygon*[linecolor=lightgray](0,0)(0,4)(1,5)(3,5)(3,2)(1,0)
\pspolygon*[linecolor=lightgray](5,0)(5,2)(6,3)(10,3)(10,2)(8,0)
\qline(0,0)(0,4)
\qline(0,4)(1,5)
\qline(1,5)(3,5)
\qline(3,5)(3,2)
\qline(3,2)(1,0)
\qline(1,0)(0,0)
\rput(1.5,2.8){$PR_{\ft}(A)$}
\rput(-0.6,0){$(0, 0)$}
\rput(0.25,5){$(1, 5)$}
\rput(3.75,2){$(3, 2)$}
\qline(5,0)(5,2)
\qline(5,2)(6,3)
\qline(6,3)(10,3)
\qline(10,3)(10,2)
\qline(10,2)(8,0)
\qline(8,0)(5,0)
\rput(7.6,1.5){$PR_{\ft}(A^T)$}
\rput(4.4,0){$(0, 0)$}
\rput(5.25,3){$(1, 3)$}
\rput(10.75,2){$(5, 2)$}
\psdots*(0,0)(1,5)(3,2)(5,0)(6,3)(10,2)
\end{pspicture}
\end{figure}

}

Again, this example extends easily to $n \geq 3$, showing that an isometry
between row spaces is not sufficient to imply that two matrices are
$\GreenJ$-related in $M_n(\ft)$ for all $n \geq 3$, and hence by
Proposition~\ref{jinherit} also in $M_n(\trop)$and $M_n(\olt)$. It also
follows, in all of these semigroups, that an isometric embedding of the
row space of $A$ into the row space of $B$ does not imply that
$A \OrdJ B$. In the case $n = 2$, it follows from results in
\cite{K_tropicalgreen} that isometry of row spaces does give an exact
characterisation of $\GreenJ$ in $M_2(\ft)$ and $M_2(\trop)$.
\end{example}

\section{$\mathcal{J}$ and the rank of a tropical matrix}

In the full matrix semigroup $M_n(K)$ of matrices over a field $K$, two matrices 
are $\mathcal{J}$-related (and hence also $\mathcal{D}$-related) if and 
only if they have the same \emph{rank}. In this classical situation there 
are several equivalent definitions of rank, stemming from the ideas of 
matrix factorisation, linear independence of rows or columns and 
singularity. Unfortunately, in the case of matrices over a semiring, these 
definitions cease to be equivalent. In this section we shall look at how 
$\mathcal{J}$-classes (and $\mathcal{D}$-classes) relate to various ideas 
of the rank of a tropical matrix. We shall show that many of the commonly 
studied ranks are $\mathcal{J}$-class invariants.

We begin by considering an arbitrary commutative semiring $S$.
We say that a function
$${\rm rank}: M_n(S) \rightarrow \mathbb{N}_0$$
is a \emph{rank function} on $M_n(S)$. We say that the 
function \textit{respects the $\GreenJ$-order} if $A \OrdJ B$ 
implies ${\rm rank}(A) \leq {\rm rank}(B)$. Clearly, any function which 
respects the $\GreenJ$-order is in particular a $\GreenJ$-class invariant, 
although the converse need not hold. We say that the function
\textit{satisfies the rank-product inequality} if and only if for all
$A, B \in M_n(S)$ we have
$${\rm rank}(AB) \leq {\rm min}({\rm rank}(A), {\rm rank}(B)).$$
The following elementary proposition observes that a rank function respects
the $\GreenJ$-order exactly if it satisfies the rank-product inequality.

\begin{proposition}\label{rankinequality} \label{rankprop}
Let $S$ be a semiring and ${\rm rank} : M_n(S) \to \mathbb{N}_0$ a rank function. Then the
function respects the $\GreenJ$-order if and only if
$${\rm rank}(AB) \leq {\rm min}({\rm rank}(A), {\rm rank}(B))$$
for all $A,B \in M_n(S)$.
\end{proposition} 

\begin{proof}
Suppose first that ${\rm rank}(AB) \leq {\rm min}({\rm rank}(A), {\rm rank}(B))$
for all $A$ and $B$. If $X \leq_{\mathcal{J}} Y$  then we may write $X=PYQ$,
where $P,Q \in M_n(S)^1$. If $P, Q \in M_n(S)$ then
\begin{eqnarray*}
{\rm rank}(X) &\leq& {\rm min}({\rm rank}(P), {\rm rank}(YQ))\\
 &\leq& {\rm min}({\rm rank}(P), {\rm rank}(Y), {\rm rank}(Q)) \leq {\rm rank}(Y).
\end{eqnarray*}
Similar arguments apply when one or both of $P$ and $Q$ is an adjoined
identity.

Conversely, suppose ${\rm rank}$ respects the $\GreenJ$-order. Then for 
$A, B \in M_n(S)$ we have $AB \OrdJ A$ and $AB \OrdJ B$, so 
${\rm rank}(AB) \leq {\rm rank}(A)$ and
${\rm rank}(AB) \leq {\rm rank}(B)$
as required.
\end{proof}

We shall see below that a number of natural notions of rank (over the tropical semiring, or 
sometimes over more general semirings) have been shown to satisfy the rank 
product inequality. It follows that all of these respect the 
$\GreenJ$-order, and hence in particular are $\GreenJ$-class invariants. 

\begin{example} (Factor Rank.)\\

Let $S$ be a commutative
semiring with addition denoted by $\oplus$ and multiplication by juxtaposition.
Let $A$ be a non-zero $m \times n$ matrix over $S$. Recall that the 
\emph{factor rank} of $A$ is the smallest natural number $k$ such that $A$ 
can be written as a product $A = BC$ with $B$ an $m \times k$ matrix and 
$C$ a $k \times n$ matrix. Equivalently, the factor rank is the smallest 
natural number $k$ such that $A$ can be written as
$$A = c_1 r_1 \oplus \cdots \oplus c_k r_k,$$
for some column $m$-vectors $c_1, \ldots, c_k$ 
and row $n$-vectors $r_1, \ldots, r_k$. Put another way, the factor rank 
of $A$ is the smallest cardinality of a set of $n$-vectors whose 
$S$-linear span contains the rows of $A$. By convention, the zero matrix
has factor rank $0$, and it is clear that no other matrix has factor rank 
$0$.
\end{example}

In the case where $S$ is a field, factor rank coincides with the usual 
definition of rank. Factor rank has been widely studied over the Boolean 
semiring (where it is called \emph{Schein rank} \cite{Kim82B}) and the 
tropical semiring (where it is sometimes called \emph{Barvinok rank} 
\cite{Develin05}). It has been observed in various semirings (see for 
example \cite[Proposition 4.4]{Beasley05}) that factor rank satisfies the 
rank-product inequality, and hence in our terminology respects the 
$\GreenJ$-order. However, this fact does not appear to have been stated
for commutative semirings in full generality; for this reason we include
a very brief proof.

\begin{corollary}
Let $S$ be a commutative semiring and $n \in \mathbb{N}$. Then factor
rank respects the $\GreenJ$-order, and hence is a 
$\GreenJ$-class invariant, in $M_n(S)$.
\end{corollary}

\begin{proof}
Let $A, B \in M_n(S)$ with $A \OrdJ B$. We shall show that the
factor rank of $A$ does not exceed that of $B$.
Note 
that if $S$ has a zero and $B$ is the zero matrix then the result holds trivially, since then 
we must have that $A=B$. Thus we may assume that $B$ is non-zero. Let
$k > 0$ be the factor rank of $B$.
Then $k$ is the smallest natural number such that we may write 
$$B= c_1 r_1 \oplus \cdots \oplus c_k r_k,$$
where each $c_i$ is a 
$n\times 1$ column vector and each $r_i$ is a $1 \times n$ row vector. 
Since $A\leq_{\mathcal{J}} B$ we have $A=PBQ$ for some matrices $P, Q \in 
M_n(S)^1$. Thus, by associativity and distributivity of matrix multiplication,
$$A=(Pc_1) (r_1Q) \oplus \cdots \oplus (Pc_k) (r_kQ),$$
where each $Pc_i$ is a 
$n\times 1$ column vector and each $r_iQ$ is a $1 \times n$ row vector. 
This gives that the factor rank of $A$ is less than or equal to $k$. In 
other words, the factor rank of $A$ is less than or equal to the factor 
rank of $B$.
\end{proof}

\begin{example} (Column and Row Rank.)\\

Let $S$ be a semiring, $n \in \mathbb{N}$ and $A$ be a non-zero matrix
in $M_{n}(S)$. The \emph{column rank} of $A$ is cardinality of the
smallest generating set for the 
column space of $A$. In the case where
$S \in \lbrace \ft, \trop, \olt \rbrace$, the column rank is simply
the generator dimension of the column 
space, as discussed in Section~\ref{sec_embednot}. The \textit{row rank} 
of $A$ is defined dually; it is shown in \cite{Akian09} that the row rank 
and column rank of a tropical matrix can differ. The zero matrix, in the
case that $S$ has a zero element, has column rank and row rank $0$.
\end{example}

Column rank and row rank are closely connected to the notion which was
called \textit{weak independence}, which was first introduced in
\cite[Chapter~16]{CuninghameGreen79}; see \cite{Butkovic10} for survey of these ideas.

In Example~\ref{ex_mutualembednotnec} of Section~\ref{sec_embednot} above we exhibited two 
$\GreenJ$-related matrices in $M_4(\trop)$ whose column spaces have 
different generator dimension. Hence column rank (and, by symmetry, row rank) 
are not $\GreenJ$-class invariants, and do not respect the 
$\GreenJ$-order. However, it follows easily from results of Hollings 
and the second author \cite{K_tropd} (quoted as part of 
Theorem~\ref{greenchar} above) that they are $\GreenD$-class invariants:

\begin{corollary}\label{Dranks}
Let $S \in \{\ft, \trop, \olt\}$, $n \in 
\mathbb{N}$ and $A, B \in M_n(S)$. If $A \mathcal{D} B$ then $A$ and $B$ 
have the same column rank and the same row rank.
\end{corollary} 

\begin{proof}
By Theorem~\ref{greenchar}(vi) and (vii), $A$ and $B$ have isomorphic 
column spaces and isomorphic row spaces, and it is immediate from the
definitions that isomorphic spaces have the same generator dimension. Thus,
$A$ and $B$ have the same column rank and the same row rank.
\end{proof}

An interesting observation is the following. If $X$ is 
a finitely generated tropical convex set (say of row vectors), then by 
Theorem~\ref{greenchar} and Corollary~\ref{Dranks}, there exists a
$k \in \mathbb{N}$ which is the column rank of \textit{every} matrix whose row 
space is $X$. In other words, $X$, in addition to its own generator dimension
as a space of row vectors, 
admits an invariant which one might call its \textit{dual dimension}. One 
might ask if this dimension manifests itself in a ``coordinate-free'' 
manner in the space $X$ itself. In fact, over the semirings $\ft$ or $\olt$,
$X$ is \textit{anti-isomorphic} to the column space of any matrix of which it
is the row space (see \cite{K_tropd}), so this dual dimension is
exactly the minimum cardinality of a generating set for $X$ under the
operations of scaling and 
\textit{greatest lower bound} within $X$. (Note that greatest lower bound
within $X$ does \textit{not} necessarily coincide with componentwise minimum, since
$X$ need not be closed under the latter operation.)

While column rank and row rank are not $\GreenJ$-class invariants in 
general (as shown for example in Example~\ref{ex_mutualembednotnec} above),
we know that $M_2(\trop)$ (by \cite{K_tropicalgreen}) and 
$M_n(\ft)$ for all $n$ (by Theorem~\ref{finiteDJ}) satisfy
$\GreenD = \GreenJ$, so in these semigroups column rank and row rank \emph{are} 
$\GreenJ$-class invariants.

\begin{example} (Gondran-Minoux Rank \cite{Gondran84}.)\\

Let $S$ be a commutative semiring with zero element. We say that $x_1, \ldots, x_t$ are 
\emph{Gondran-Minoux independent} if whenever
$$\sum_{i \in I} \alpha_i x_i = \sum_{j \in J} \alpha_j x_j,$$
with $I,J \subset \{1, \ldots, t\}$, $I \cap J = \emptyset$ and
$\alpha_1, \ldots, \alpha_t \in S$ it follows that $\alpha_1, \ldots, \alpha_t$ are all zero.
Now let $A \in M_{n}(S)$. The \emph{maximal Gondran-Minoux column [row] rank} 
of $A$ is the maximal number of Gondran-Minoux independent columns [rows] of
$A$.
\end{example}

Note that this notion of rank is explicitly defined in terms of the actual 
columns of $A$, rather than just the column space, so there
is no immediate reason to suppose that it is a column space invariant (or 
equivalently by Theorem~\ref{greenchar}(iv), an $\GreenR$-class 
invariant). Rather surprisingly, however, recent results of Shitov 
\cite{Shitov10} imply that maximal Gondran-Minoux column rank even
respects the $\GreenJ$-order 
in full matrix semigroups over $\trop$ and $\olt$. Indeed, in
\cite{Shitov10} it is shown that for a class of semirings including
$\trop$ and $\olt$ (specifically, idempotent
\textit{quasi-selective} semirings with zero and without zero divisors),
this notion of rank satisfies the rank-product inequality, so
Proposition~\ref{rankprop} yields:

\begin{corollary}
Let $S \in \lbrace \trop, \olt \rbrace$ and $n \in \mathbb{N}$.
Then maximal Gondran-Minoux column rank and
maximal Gondran-Minoux row rank respect the $\GreenJ$-order, and hence are 
$\GreenJ$-class invariants, in $M_n(S)$.
\end{corollary}

It is interesting that, although the Gondran-Minoux ranks give 
$\mathcal{J}$-class invariants in both $\trop$ and $\olt$, the rank 
functions themselves are dependent upon which tropical semiring one
works in (whilst the $\mathcal{J}$-classes are not, by Proposition 
\ref{jinherit}). For instance, the matrix
$$A=\left(\begin{array}{c c c c} 0& 0\\ 0&1 \end{array}\right)$$
has maximal Gondran-Minoux column rank $2$ over $\trop$, but $1$ over $\olt$.

\begin{example} (Determinantal rank.)\\

Another equivalent way to define rank in classical linear algebra is 
as the dimension of the largest non-singular submatrix. In a 
semiring, without negation, it is not entirely clear how to define
singularity. However, one reasonable approach is to regard a
matrix as singular if the terms which would 
normally have positive coefficients in the determinant have the same sum 
as the terms as those which would normally have negative coefficients.
\end{example}

More formally, for a $k \times k$ matrix $M$ over a commutative semiring, we define 
\begin{eqnarray*}
|M|^+ &=& \sum_{\substack{\sigma \in S_k,\\ {\rm sign}(\sigma)=1}} m_{1,\sigma(1)} \cdots m_{k, \sigma(k)}\\
|M|^- &=&\sum_{\substack{\sigma \in S_k,\\ {\rm sign}(\sigma)=-1}} 
m_{1,\sigma(1)} \cdots m_{k, \sigma(k)},
\end{eqnarray*}
where $S_k$ 
denotes the symmetric group on $\{1, \ldots, k\}$. The \emph{determinantal 
rank} of $A$ is the largest integer $k$ such 
that $A$ has a $k \times k$ minor $M$ with $|M|^+ \neq |M|^-$. In 
\cite[Theorem 9.4]{Akian09} it was shown that determinantal rank over
$\trop$ (and hence also over $\ft$) satisfies the rank-product inequality,
which combined with Proposition~\ref{rankprop} yields:

\begin{corollary}
Let $S \in \lbrace \ft, \trop \rbrace$ and $n \in \mathbb{N}$. Then
determinantal rank respects the $\GreenJ$-order, and hence is a
$\GreenJ$-class invariant, in $M_n(S)$.
\end{corollary}

\begin{example} (Tropical Rank)\\

Another natural, and frequently used, notion of singularity for tropical
matrices 
is the following. For $S \in \{\ft, \trop, \olt\}$ we define a matrix
$M \in M_k(S)$ to be \emph{strongly regular} if there is no non-empty 
subset $T \subseteq S_k$ such that
$$\bigoplus_{\sigma \in T} m_{1,\sigma(1)} \cdots m_{k, \sigma(k)} = \bigoplus_{\sigma \in S_k \setminus T} m_{1,\sigma(1)} \cdots m_{k, \sigma(k)}.$$
Now for $A \in M_n(S)$, the 
\emph{tropical rank} of $A$ is the largest integer $k$ such that $A$ has
a strongly regular $k \times k$ minor.
\end{example}

Strongly regular matrices and tropical rank were first studied in
\cite[Chapters 16~and~17]{CuninghameGreen79}, where they are called
just \textit{regular} matrices and
\textit{rank} respectively. A number of equivalent formulations have since 
been discovered (see for example 
\cite{Akian09,Butkovic10,Develin05,Izhakian09}). Perhaps most 
interestingly, over $\ft$ tropical rank is the maximum topological dimension of the 
row (or column) space viewed as a subset of $\mathbb{R}^n$ with the usual 
topology \cite[Theorem~4.1]{Develin05}. In \cite[Theorem 9.4]{Akian09} it 
was shown that tropical rank of matrices over $\trop$ (and hence also over 
$\ft$) satisfies the rank-product inequality. Combining with 
Proposition~\ref{rankprop} we have:

\begin{corollary}
Let $S \in \lbrace \ft, \trop \rbrace$ and $n \in \mathbb{N}$. Then tropical 
rank respects the $\GreenJ$-order, and hence is a $\GreenJ$-class 
invariant, in $M_n(S)$.
\end{corollary}

Finally, we remark briefly that there are a number of other notions of 
rank for tropical matrices (see for example \cite{Akian09}), and we do not 
claim that the study presented in this section is exhaustive. One which 
has proved to be of interest for applications in algebraic geometry is 
\emph{Kapranov rank} (see for example \cite{Develin05}); its relationship 
with Green's relations deserves detailed study.

\section*{Acknowledgements}

This research was supported by EPSRC grant number EP/H000801/1
(\textit{Multiplicative Structure of Tropical Matrix Algebra}).
The second author's research is also supported by an RCUK
Academic Fellowship. The authors thank Peter Butkovic, Zur Izhakian
and Stuart Margolis for helpful conversations.

\bibliographystyle{plain}

\def\cprime{$'$} \def\cprime{$'$}

\end{document}